\documentclass[a4paper,10pt]{amsart}
\usepackage[english]{babel}
\usepackage[T1]{fontenc}
\usepackage{graphicx}
\usepackage{tikz, pgfplots}
\usepackage{hyperref}
\usepackage{authblk}
\usepackage{amssymb}
\usepackage{amsmath}
\usepackage{amsfonts}
\usepackage{amsthm}
\usepackage{amscd}
\usepackage{mathrsfs}
\usepackage[foot]{amsaddr}

\theoremstyle{plain}
\newtheorem{thm}{Theorem}[section]
\newtheorem{prop}[thm]{Proposition}

\newtheorem{cor}[thm]{Corollary}
\newtheorem{rmk}[thm]{Remark}

\newtheorem{exm}[thm]{Example}

\theoremstyle{definition}
\newtheorem{defin}[thm]{Definition}

\newcommand{\K}{\mathbb{K}}

\newcommand{\p}{\mathbb{P}}
\newcommand{\B}{\mathcal{B}}

\newcommand{\depa}[2]{\frac{\partial #1}{\partial #2}}

\DeclareMathOperator{\hess}{hess}
\DeclareMathOperator{\Hess}{Hess}
\DeclareMathOperator{\ann}{Ann}
\DeclareMathOperator{\Hilb}{Hilb}

\DeclareMathOperator{\Hom}{Hom}

\title{ Lefschetz properties for Higher order Nagata idealizations}

\author{Armando Cerminara*}
\address{*Dipartimento Matematica ed Applicazioni ``Renato Caccioppoli", Universit\`a Degli Studi Di Napoli ``Federico II", Via Cinthia - Complesso Universitario Di Monte S. Angelo 80126 - Napoli - Italia}
\email{armando.cerminara@unina.it}

\author[R. Gondim]{Rodrigo Gondim**}
\address{**Universidade Federal Rural de Pernambuco, av. Don Manoel de Medeiros s/n, Dois Irmãos - Recife - PE
52171-900, Brasil}
\email{rodrigo.gondim@ufrpe.br}

\author{Giovanna Ilardi*}
\email{giovanna.ilardi@unina.it}

\author{Fulvio Maddaloni*}
\email{maddalonifulvio@gmail.com}

\subjclass[2010]{Primary 13A02, 05E40; Secondary 13D40, 13E10}
\keywords{Lefschetz properties, Artinian Gorenstein Algebras, Nagata idealization.}
\date{}

\begin{document}

\begin{abstract}
We study a generalization of Nagata idealization for level algebras. These algebras are standard graded Artinian algebras whose Macaulay dual generator is given explicitly as a bigraded polynomial of bidegree $(1,d)$. We consider the algebra associated to polynomials of the same type of bidegree $(d_1,d_2)$. We prove that the geometry of the Nagata hypersurface of order $e$ is very similar to the geometry of the original hypersurface. We study the Lefschetz properties for Nagata idealizations of order $d_1$, proving that WLP holds if $d_1 \geq d_2$. We give a complete description of the associated algebra in the monomial square free case.  
\end{abstract}

\thanks{**Partially supported  by ICTP-INdAM Research in Pairs Fellowship 2018/2019 and by FACEPE ATP - 0005-1.01/18.}

\maketitle

\section*{ Introduction}

The Lefschetz properties are algebraic abstractions inspired by the so called Hard Lefschetz Theorem about the cohomology of smooth projective varieties over the complex numbers (see \cite{La,Ru}). Since the cohomology ring of such varieties are standard graded
Artinian $\K-$algebras satisfying the Poincar\'e duality, from the algebraic viewpoint these algebras can be characterized as standard graded Artinian Gorenstein algebras, AG algebras for short (see \cite{HMMNWW, MW}). In this context, $A =\displaystyle \bigoplus_{k=0}^d A_k$, the Hard Lefschetz Theorem can be reformulated as a possible purely algebraic property of the algebra $A$.\\ 

It is important to highlight that nowadays the Lefschetz properties are considered in a number of distinct contexts, such as Khaler manifolds, solvmanifolds (see \cite{Ka}), arithmetic hyperbolic manifolds (see \cite{Be}), Shimura varieties (see \cite{HL}), convex polytopes (see \cite{KN}), Coxeter groups (see \cite{NW}), matroids, simplicial complexes \cite{St,St2,BN,GZ,KN} among others. 
In these new contexts the Lefschetz properties showed to have interactions with the algebra itself, the geometry and the combinatorics. This work lies at the intersection of these three areas. In fact, standard graded AG algebras can be presented as $A = Q/\ann(f)$ where $f \in \K[x_0,\ldots,x_n]_d$ is a homogeneous form of degree $d$ in a polynomial ring, $Q = \K[\partial_0,\ldots,\partial_n]$ is the associated ring of differential operators and $\ann(f)$ is the ideal of differential \-operators that annihilates $f$. Therefore, in this paper we are interested in the algebraic structure of $A$, together the geometry of the hypersurface $X = V(f) \subset \p^{n}$ and the combinatorics of the form $f$ in a very particular way.  

The cornerstone of the algebraic theory of Lefschetz properties were the original\- papers of Stanley \cite{St, St2, St3} and the works of Watanabe, summarized~ in \cite{HMMNWW}. A very important construction that appears many times in these works is the so called Nagata idealization also called trivial extension. In \-general Nagata idealization is a useful tool, developed by Nagata, to convert any $R-$module $M$ in a ideal of another ring, $A \ltimes M$. In our perspective the starting point is a very interesting isomorphism between the Nagata idealization of an ideal $I=(g_0,\ldots,g_m) \subset \K[u_1,\ldots,u_m]$ and a level algebra in such way that the new ring is an AG algebra and we get an explicit formula for the Macaulay generator $f$ (see \cite[Proposition 2.77]{HMMNWW})
\begin{equation}\label{eq:perazzo}
 f=x_0g_0+\ldots+x_ng_n \in \K[x_0,\ldots,x_n,u_1,\ldots,u_m]_{(1,d-1)}.
\end{equation}
This bigraded polynomial is closely related with Gordan-Noether and Perazzo constructions of forms with vanishing Hessian (see \cite{GN, Pe, CRS, GRu}). It is not a coincidence since in \cite{MW} the authors present a Hessian criterion 
for the SLP saying that the vanishing of a (higher) Hessian implies the failure of SLP. This criterion was generalized in \cite{GZ2} also for the WLP using mixed Hessians. Following the original ideas of Gordan-Noether and Perazzo, the second author in 
\cite{Go} constructed families of polynomials whose $k$-th Hessian is zero. 
A natural generalization of (\ref{eq:perazzo}) should be to consider polynomials of the form: 
\begin{equation}\label{eq:perazzo1}
 f=x_0^{d_1}g_0+\ldots+x_n^{d_1}g_n \in \K[x_0,\ldots,x_n,u_1,\ldots,u_m]_{(d_1,d_2)}.
\end{equation}
These polynomials are called Nagata polynomials of order $d_1$ (see Definition \ref{sftdefn}). 

The study of the hypersurfaces with vanishing Hessian began in 1852, when O.~Hesse wrote two papers  (see \cite{He1851, He1859}), in which he claimed they must be cones. Given an irreducible hypersurface $X=V(f)\subset \p^N$ of degree $\deg f=d\geq 3$, P. Gordan and M. Noether proved in \cite{GN} that Hesse's claim is true for $N\leq 3$, but it is false for $N\geq 4$; in fact it is possible to construct counterexamples in $\p^N$ for each $N\geq 4$ and for all $d\geq 3$. Moreover, they classified all the counterexamples to Hesse's claim in $\p^4$ (see \cite{CRS, Go, GN}). Perazzo classified cubic hypersurfaces with vanishing Hessian in $\p^N$ for $N \leq 6$ (see \cite{Pe}), this work was 
revisited and generalized in \cite{GRu}. The problem is open in the other cases. In all the cases where the classification of hypersurfaces with vanishing Hessian are done they share two very particular geometric properties (see \cite{CRS, GR,GRu, Ru}):
\begin{enumerate}
 \item[(i)] there is a linear space $L$ in the singular locus of $X$, that is the linear span of the dual variety of the image of the polar (gradient) map, that is $L = <Z^*>$;
 \item[(ii)] the hypersurface is a tangent scroll over the dual of the polar image.
\end{enumerate}

\medskip

In this paper we study the Lefschetz properties for the algebras associated to Nagata polynomials of order $d_1$, the geometry of the Nagata hypersurfaces of order $d_1$ and the interaction between the combinatorics of $f$ and the algebraic structure  
of $A$ in the case that the $g_i$ are square free monomials, by using a simplicial complex to study this case. \\
\break 

We show that the  geometry of Nagata hypersurfaces is very similar to the geometry\- of the known hypersurfaces with vanishing Hessian. Hence these are hypersurfaces, satisfying at least a Laplace equation (see \cite{DI, DDI}). We proved that they are scroll hypersurfaces in Theorem \ref{geometrictheorem} and Corollary \ref{corollarygeometric}. These are our first main results. \\

From the algebraic viewpoint we are interested in the Lefschetz properties and the algebraic structure of Nagata idealizations. The Lefschetz properties are studied in two cases:
\begin{enumerate}
 \item $d_1<d_2$, in this case we give examples with small numbers of summands where the SLP holds and we recall a result proved in \cite{Go} (see Proposition \ref{GNPproposition});
 \item $d_1\geq d_2$, in this case $A$ has the WLP as proved in Proposition \ref{prop:egrande}. This is our second main result. 
\end{enumerate}

The structure of the algebra $A$, including the Hilbert vector and a complete description of the ideal $I$ was proved in Theorem \ref{teoremaimportante} for the case in which the $g_i$ are square free monomials. This is our third main result. To prove it, we use the combinatorics of the simplicial complex associated to the monomials $g_i$ in order to describe both $A_k$ and $I_k$.\\

Part of this paper was inspired by the discussions of a group work in the workshop {\it Lefschetz Properties and Artinian Algebras} at BIRS, Banff, Canada in Mach, 2016. 
The participants of the group work were M. Boij, R. Gondim, J. Migliore, U. Nagel, A. Seceleanu, H. Schenck and J. Watanabe.

\section{Artinian Gorenstein algebras and the Lefschetz properties}

In this section we recall some basic facts about Artinian Gorenstein algebras and the Lefschetz properties. For a more detailed account, let see \cite{HMMNWW, MN1,  Ru, MW, Go}.

\subsection{Standard graded Artinian Gorenstein algebras and Hilbert vector.}

In all the paper $\K$ denotes a field of characteristic zero. 
\begin{defin}\rm Let $R = \K[x_0,\ldots,x_n] $ be the polynomial ring in $n+1$ variables and $I \subset R$ be an homogeneous Artinian ideal such that $I_1=0$. 
We say that a graded Artinian $\K-$algebra $A=R/I = \displaystyle\bigoplus_{i=0}^dA_i$ is a standard graded Artinian $\K-$algebra  if it is generated in degree $1$ as algebra. Setting $h_i=\dim_\K A_i$, the \emph{Hilbert vector} is $\Hilb(A)=(1,h_1,\dots,h_d)$. If $I_1=0$, then $h_1$ is called the codimension of $A$.
\end{defin}

\begin{defin}\rm
A standard graded Artinian algebra $A$ is Gorenstein if and only if $\dim A_d = 1$ and the restriction of the multiplication of the algebra in complementary degree, that is, $A_k \times A_{d-k} \to A_d$ is a perfect paring for $k =0,1,\ldots,d$ (see \cite{MW}). If $A_j=0$ for $j >d$, then $d$ is called the socle degree of $A$.
\end{defin}

\begin{rmk}\rm
Since  $A_k \times A_{d-k} \to A_d$ is a perfect paring for $k =0,1,\ldots,d$, it induces two $\K-$linear maps, $A_{d-k}\to A_k^*$, with $A_k^*:=\Hom(A_k,A_{d})$ and $A_k\to A_{d-k}^*$, with $A_{d-k}^*:=\Hom(A_{d-k}, A_d),$ that are two isomorphisms.  
\end{rmk}
Let $R=\K[x_0,\ldots, x_n]$ be the polynomial ring in $n+1$ variables. We denote by $R_d=\K[x_0,\ldots,x_n]_d$ the $\K-$vector space of homogeneous polynomials of degree $d$.\\ We denote by $Q=\K[X_0,\ldots,X_n]$ the ring of differential operators of $R$, where $X_i := \depa{}{x_i}$ for $i=0,\ldots,n.$ We denote by $Q_k=Q[X_0,\ldots,X_n]_k$ the $\K-$vector space of homogeneous differential operators of $R$ of degree $k$.\\ For each $d\geq k\geq 0$ there exist natural $\K-$bilinear maps $R_d\times Q_k \to R_{d-k}$ defined by differentiation: $$(f,\alpha) \to f_\alpha := \alpha(f).$$
Let $f\in R$ be a homogeneous polynomial of degree $\deg f=d\geq 1$, we define: $$\ann (f) :=\left\{\alpha\in Q | \alpha(f)=0\right\}\subset Q.$$ This is called the \emph{annihilator of $f$}.\\

Since $\ann(f)$ is a homogeneous ideal of $Q$, we can define $$A=\frac{Q}{\ann(f)}.$$ $A$ is a standard graded Artinian Gorenstein $\K-$algebra.\\


Conversely, by the theory of inverse systems, we get the following characterization of standard graded
Artinian Gorenstein $\K-$algebras. 

\begin{thm}{\bf \ (Double annihilator Theorem of Macaulay)} \label{G=ANNF} \\
Let $R = \K[x_0,\ldots,x_n]$ and let $Q = \K[X_0,\ldots, X_n]$ be the ring of differential operators. 
Let $A= \displaystyle \bigoplus_{i=0}^dA_i = Q/I$ be an Artinian standard graded $\mathbb K-$algebra. Then
$A$ is Gorenstein if and only if there exists $f\in R_d$
such that $A\simeq Q/\operatorname{Ann}(f)$.
\end{thm}

A proof of this result can be found in \cite[Theorem 2.1]{MW}.

\begin{rmk} \rm With the previous notation, let $A= \displaystyle \bigoplus_{i=0}^dA_i = Q/I$ be an Artinian Gorenstein $\K-$algebra with $I = \ann(f)$, $I_1=0$ and $A_d \neq 0$. The socle degree of $A$ coincides with the degree of the form $f$. 
\end{rmk}

Now we deal with standard bigraded Artinian Gorenstein algebras, i.e. Artinian Gorenstein algebras, $A=\displaystyle\bigoplus_{i=0}^d A_i$, such that 
$$\begin{cases}
A_d\neq 0 \\
A_k=\displaystyle\bigoplus_{i=0}^kA_{(i,k-i)} \mbox{ for } k<d 
\end{cases}.$$
The pair $(d_1,d_2)$, such that $A_{(d_1,d_2)}\neq 0$ and $d_1+d_2=d,$  is said the socle bidegree of $A$.

\begin{rmk} \rm
 Since $A_k^*\simeq A_{d-k}$ and since duality is compatible with direct sum, we get $A^*_{(i,j)}\simeq A_{(d_1-i,d_2-j)}$.
 \end{rmk}

By abuse notation, we denote the polynomial ring viewed as standard bigraded ring in the set of variables $\left\{x_0,\ldots,x_n\right\}$ and $\left\{u_1,\ldots,u_m\right\}$ by $R=\K[x_0,\ldots,x_n,u_1,\ldots,u_m]$. A homogeneous polynomial $f\in R_{(d_1,d_2)}$ is said to be bihomogeneous polynomial of total degree $\deg f=d=d_1+d_2$ if  $f$ can be written in the following way: \begin{equation}\label{hp}f=\displaystyle\sum_{i=1}^s f_i g_i,\end{equation} where $f_i\in \K[x_0,\dots,x_n]_{d_1}$ and  $g_i\in\K[u_1,\dots,u_m]_{d_2},\forall i\le s$. 
\begin{rmk} \rm
All bihomogeneous polynomials $f\in\K[x_0,\dots,x_n,u_1,\dots u_m]_{(d_1,d_2)}$ can be written as (\ref{hp}), where $f_i\in\K[x_0,\ldots,x_n]_{d_1}$ and $g_i\in\K[u_1,\dots,u_m]_{d_2},\forall i\le s$, are monomials.\end{rmk}   A homogeneous ideal $I\subset R$ is a bihomogeneous ideal if $$I=\displaystyle\bigoplus_{i,j=0}^\infty I_{(i,j)}$$ where $I_{(i,j)}= I\cap R_{(i,j)}$ $\forall i,j$. Let $Q=\K[X_0,\ldots,X_n,U_1,\ldots,U_m]$ be the associated ring of differential operators and let $f\in R_{(d_1,d_2)}$ be a bihomogeneous polynomial of total degree $d=d_1+d_2$, then $I=\ann(f)\subset Q$ is a bihomogeneous ideal and $A=Q/I$ is a standard bigraded Artinian Gorenstein algebra of socle bidegree $(d_1,d_2)$ and codimension $N=n+m+1$. 

\begin{rmk}\label{rmk1.8} \rm
 Let $f\in R_{(d_1,d_2)}$ be a bihomogeneous polynomial of degree $(d_1,d_2)$, and let $A$ be the associated bigraded algebra of socle bidegree $(d_1,d_2),$ then for $i>d_1$ or $j>d_2$:$$I_{(i,j)}=Q_{(i,j)}.$$ In fact for all $\alpha\in Q_{(i,j)}$ with $i>d_1$ or $j>d_2$ we get $\alpha(f)=0$, so $Q_{(i,j)}=I_{(i,j)}.$ As consequence, we have the following decomposition for all $A_k$: $$A_k=\displaystyle\bigoplus_{i\leq d_1,j\leq d_2,i+j=k}A_{(i,j)}.$$
Furthermore for $i<d_1$ and $j< d_2$, the evaluation map $Q_{(i,j)}\to A_{(d_1-i,d_2-j)}$ given by $\alpha\to\alpha(f)$ provides the following short exact sequence:
$$\begin{CD}
0@>>>I_{(i,j)}@>>>Q_{(i,j)}@>>>A_{(d_1-i,d_2-j)}@>>>0\end{CD}.$$\end{rmk}


\subsection{The Lefschetz properties and the Hessian criterion}




\begin{defin}\label{WLPSLP}\rm 
Let $$A=\bigoplus_{i=0}^dA_i$$ be an Artinian graded $\mathbb K-$algebra with $A_d\neq 0$.

The algebra  $A$ is said to have {\it the Weak Lefschetz Property}, briefly $WLP$, if there exists an element $L\in A_1$ such that the multiplication map
$$\bullet L: A_i\to A_{i+1}$$

is of maximal rank for $0\leq i\leq d-1$.

\medskip

The algebra  $A$ is said to have {\it the Strong Lefschetz Property}, briefly $SLP$, if there exists an element $L\in A_1$ such that the multiplication map
$$\bullet L^k: A_i\to A_{i+k}$$

is of maximal rank for $0\leq i\leq d$ and $0\leq k\leq d-i$.


 
 $A$ is said to have {\it the Strong Lefschetz Property in the narrow sense} if there exists an element $L\in A_1$ such that the multiplication map
$$\bullet L^{d-2i} : A_i \to A_{d-i}$$
 is bijective for $i = 0,\ldots,\lfloor\frac{d}{2}\rfloor.$
 \end{defin}

\begin{rmk}\rm In the case of standard graded Artinian Gorenstein algebras the two condition SLP and SLP in the narrow sense are equivalent. 
 
\end{rmk}

\begin{defin}\rm
Let $f\in R_d$ be a homogeneous polynomial, let $A=\displaystyle\bigoplus_{i=0}^dA_i=\frac{Q}{\ann(f)}$ be the associated Artinian Gorenstein algebra and let $\B=\left\{\alpha_j| j=1,\ldots,\sigma_k\right\}\subset A_k$ be an ordered $\K-$basis of $A_k$. The \emph{$k-$th  Hessian matrix} of $f$ with respect to $\B$ is $$\Hess_f^k:=\left(\alpha_i\alpha_j(f)\right)_{i,j=1}^{\sigma_k}.$$ The \emph{$k-$th Hessian} of $f$ with respect to $\B$ is $$\hess_f^k:=\det(\Hess_f^k).$$
\end{defin}


\begin{thm}\label{WHessian} {\rm (\cite{Wa1}, \cite{MW})} Let notation be as above.
An element $L = a_1X_1+\ldots+a_nX_n\in A_1$ is a strong Lefschetz element of $A = Q/\operatorname{Ann}(f)$ if and only if $\hess^k_f(a_1,\ldots, a_n)\neq 0$ for all $k=0,\ldots, \lfloor d/2\rfloor$.
In particular, if for some $k\leq \lfloor\frac{d}{2}\rfloor$ we have $\hess_f^k =0$, then $A$ does not have the SLP.
\end{thm}


\section{Higher order Nagata idealization}

\subsection{Nagata idealization}

\begin{defin}\rm Let $A$ be a ring and $M$ be a $A-$module. The idealization of $M$, $A \ltimes M$, is the product set $A \times M$ in which addition and multiplication are defined as follows: $$(a,m)+(b,n)=(a+b,m+n) \mbox{ and } (a,m).(b,n)=(ab,bm+an).$$ 
 
\end{defin}

The following is a known result whose proof can be found in \cite[Theorem 2.77]{HMMNWW}.

\begin{thm}\label{thm:isomorphism}
 Let $R=\K[u_1,\ldots,u_n]$ and $R'=\K[u_1,\ldots,u_n,x_0,\ldots,x_n]$ be polynomial rings and let $Q=\K[\partial_1,\ldots,\partial_n]$ and $Q'=\K[\partial_1,\ldots,\partial_n,\delta_0,\ldots,\delta_n]$ the associated ring of differential operators. 
Let $I=(g_1,\ldots,g_m) \subset Q$ be an ideal generated by forms of degree $d$ and let $A=Q/Ann(g_1,\ldots,g_m)$ be the associated level algebra. Let $f=x_0g_0+\ldots+x_mg_m \in R'$ be a bihomogeneous polynomial and let $A' = Q'/Ann(f)$ be the associated algebra. Considering $I$ as an $A-$module, we have $$A \ltimes I \simeq A'$$
\end{thm}

\subsection{Lefschetz properties for higher order Nagata idealization}

\begin{defin}\label{sftdefn}\rm
A bihomogeneous polynomial 
\begin{equation}\label{eq:5}
f=\displaystyle\sum_{i=0}^s x_i^{d_1}g_i\in\K[x_0,\ldots,x_n,u_1,\ldots,u_m]_{(d_1,d_2)}\end{equation}
is called a Nagata polynomial of order $d_1$, if the polynomials $g_i$ are linearly independent and they depend on all variables.\end{defin}
By Theorem \ref{thm:isomorphism}, the algebra $A=Q/\ann(f)$ can be realized as a trivial extension and it is said Nagata idealization of order $d_1$, socle degree $d_1+d_2$ and codimension $n+m+1$.

Let $R=\K[x_0,\ldots,x_n,u_1,\ldots,u_m]$ be the polynomial ring and $f\in R_{(d_1,d_2)}$, with $d_1\geq 1$, be a polynomial of type $f=\displaystyle\sum_{i=0}^n x_i^{d_1} g_i$, where $g_i$ is a polynomial in $u_1,\ldots,u_m$ variables, for all $i=0,\ldots,m$. We denote by $Q=\K[X_0,\ldots,X_n,U_1,\ldots,U_m]$  the ring of differential operators of $R$, where $X_i=\depa{}{x_i}$, for $i=0,\ldots,n$ and $U_j=\depa{}{u_j}$, for $j=1,\ldots,m$. Let $A=\frac{Q}{\ann (f)}$ the associated algebra.

In the case $d_1< d_2$, we have an example such that $A$ has the SLP, hence $A$ has the WLP:
 
 \begin{exm}\rm Let $f=x^2u^3+y^2v^3$ be a bihomogeneous polynomial. Hence $A$ has bidegree $(2,3)$, Hilbert vector $(1,4,6,6,4,1)$ and $A$ has the SLP.
  By the Hessian criterion, Theorem \ref{WHessian}, there are two Hessians to control, $\hess^1_f\neq 0$ and $\hess^2_f\neq 0$. 
  \end{exm}

If the number of summands in $f$ is great enough, we get the following Proposition:
 
 \begin{prop}[\cite{Go}, Proposition 2.5]\label{GNPproposition}
Let $x_0,\ldots,x_n$ and $u_1,\ldots,u_m$ be independent sets of indeterminates with $n\geq m\geq 2$. For $j=1,\ldots,s,$ let $f_j\in\K[x_0,\ldots,x_n]_{d_1}$ and $g_j\in\K[u_1,\ldots,u_m]_{d_2}$ be linearly independent forms with $1\leq d_1 < d_2$. If $s> \binom{m-1+d_1}{d_1}$, then the form of degree $d_1+d_2$ given by 
$$ f=f_1g_1+\dots+f_sg_s $$
satisfies $$\hess^k_f=0$$
\end{prop}

\begin{cor}
 Let $A$ be a Nagata idealization of order $d_1<d_2$, then $A$ fails SLP.
\end{cor}

 If we consider $d_1 \geq d_2$, we have the following Proposition:

\begin{prop}\label{prop:egrande} With the same notations, if $d_1 \geq d_2$, then $A$ has the WLP and $L=\displaystyle\sum_{i=0}^nX_i$ is a weak Lefschetz element. 
\end{prop}

\begin{proof} (The idea of this result was shared by the work group in Banff).\\
We denote by $k=\lfloor \frac{d_1+d_2}{2}\rfloor$. We note that $d_1\geq k$. Infact, by hypothesis $d_1\geq d_2$, hence: $$d_1+d_1\geq d_1+d_2\Rightarrow \frac{2d_1}{2}\geq \frac{d_1+d_2}{2}\Rightarrow d_1\geq \frac{d_1+d_2}{2}\geq \lfloor \frac{d_1+d_2}{2}\rfloor=k.$$ We have: $$A_k=A_{(k,0)}\oplus A_{(k-1,1)}\oplus\cdots\oplus A_{(k-d_2,d_2)}.$$ We want to prove that for $L = X_0+\ldots+X_n \in Q[X_0,\ldots,X_n]_1$ $$\bullet L\colon A_{(k-i,i)}\to A_{(k-i+1,i)}$$ has maximal rank for all $i=0,\ldots,d_2$. Since $A$ is a standard graded AG algebra it is enough to check it in the middle (see \cite{MMN}, Proposition 2.1).  
\\
We denote $\omega_{j}=X_j^{k-i}\alpha_{j}$, where $\alpha_{j}\in Q[U_1,\ldots,U_m]_i$, for $j=0,\ldots,n$ and we suppose that $\left\{\omega_{j}\right\}$ is a basis for $A_{(k-i,i)}.$ Hence we get $$\displaystyle\sum_{j}b_j\omega_j=0\Rightarrow b_j=0. $$
It implies that the $\alpha_j(g_j) $ are linear independent in $\K(x_1,\ldots,x_n)$. \\
Let $\Omega_{j}=X_j^{k-i+1}\alpha_{j}=\bullet L(\omega_j)$, we want to prove that $\{\Omega_{0}, \ldots, \Omega_{n}\}$ is a linear independent system for $A_{(k-i+1,i)}$. We consider the following linear combination $\displaystyle\sum_{j}c_j\Omega_j=0$. By definition, we get: 
$$0=\displaystyle\sum_jc_j\Omega_j(f)=\displaystyle\sum_jc_j\Omega_j\left(\displaystyle\sum_ix_i^{d_1}g_i\right)=\displaystyle\sum_jc_jx_j^{d_1-k+i-1}\alpha_j(g_j).$$  Since $\alpha_j(g_j)$ are linear independent in $\K(x_1,\ldots,x_n),$ for all $j=0,\ldots,n$, we have $$c_jx_j^{d_1-k+i-1}=0\Rightarrow c_j=0.$$ The result follows.
\end{proof}
For  this case, there is nothing we are able to say about the SLP.
\subsection{The geometry of Nagata hypersurfaces of order $d_1$}

\begin{defin}\label{GNPdefn}\rm Let $R= \K[x_0,\ldots,x_n,u_1,\ldots,u_m]$ be the polynomial ring, with $\K$ an algebraically closed field. Let $f\in R$ be a Nagata polynomial of order $d_1$ and degree $\deg f=d=d_1+d_2$.

The hypersurface $X=V(f)\subset \p^N$ is called a Nagata hypersurface of order $d_1$. 
\end{defin}

Let $X=V(f)\subset\p^N$ be a Nagata hypersurface of order $d_1$. We can consider two linear space respectively $\p^{m-1}$ with coordinates $u_1,\ldots,u_m$ and $\p^n$ with coordinates $x_0,x_1,\dots,x_n$.  Let $p_\alpha\in \p^{m-1}$ be a point and we consider the following linear space of dimension $n+1$: $$\mathcal{L}_\alpha:=\left<p_\alpha, \p^n \right>=\left\{\left<p_\alpha, q\right> : q\in \p^n\right\}.$$ 
If we consider the intersection $\mathcal{L}_\alpha$ with $X$, we obtain a variety $Y_\alpha$. $Y_\alpha$ is reducible whose irreducible components are the linear space $\p^n$ and a variety, called \emph{residue} and denoted by $\tilde{Y}_\alpha$. $\tilde{Y}_\alpha$ is a cone of vertex $p_\alpha$ over a $(n-1)-$dimensional basis.

\begin{thm}\label{geometrictheorem}
A Nagata hypersurface $X=V(f)\subset\p^N$ of order $e$ consists of the union of the residue parts $\tilde{Y}_\alpha$, i.e.  $$X=\displaystyle\cup_\alpha \tilde{Y}_\alpha.$$
\end{thm}
\begin{proof}
Fixed a point $p_\alpha=(0:\ldots:0:a_1:\ldots:a_m)\in\p^{m-1}$ and let $\overline{p}=(\overline{x_0}:\ldots:\overline{x_n}:0:\ldots:0)$ be a point in $\p^n.$ We consider the line that joins the points $p_\alpha$ and $\overline{p}:$
$$\mathscr{L}_\alpha: \begin{cases}
x_0=\lambda \overline{x_0}\\
\cdots\cdots\cdots\\
x_n=\lambda\overline{x_n}\\
u_1=\mu a_1\\
\cdots\cdots\cdots\\
u_m=\mu a_m 
\end{cases}$$ with $\lambda,\mu\in \K.$ \\ 
Since $X=V(f)$ is a Nagata hypersurface of order $d_1$, we have:

$$f=x_0^{d_1}g_0+\ldots+x_n^{d_1}g_n.$$


 If we consider the intersection between the line $\mathscr{L}_\alpha$ and the Nagata hypersurface $X$, we get:
$$
f_{\small{\mathscr{L}_\alpha}}=\lambda^{d_1}\overline{x_0}^{d_1}g_0(\mu a_1,\ldots,\mu a_m)+\ldots + \lambda^{d_1} \overline{x_n}^{d_1}g_n(\mu a_1,\ldots, \mu a_m)=\lambda^{d_1}\mu^{d_2}\displaystyle\sum_{i=0}^n \overline{x_i}^{d_1}g_i(\underline{a})
$$


where $\underline{a}$ is the vector $(a_1,\ldots,a_m)$. 
\par Since $p_\alpha$ and $\overline{p}$ are points of $X$, then $\displaystyle\sum_{i=0}^n\overline{x_i}^{d_1}g_i(\underline{a})=0$.
Therefore $$\tilde{Y}_\alpha=V\left(\displaystyle\sum_{i=0}^n \overline{x_i}^{d_1}g_i(\underline{a})\right)$$


 and, by arbitrariness of the points $p_\alpha\in\p^{m-1}$ and $\overline{p}\in\p^n$, we have $\displaystyle\cup_{\alpha}\tilde{Y}_\alpha=X.$
\end{proof}
 As consequence of the above theorem, we can say how many linear spaces there are in a Nagata hypersurface of order $e$. We note that $\p^{m-1}$ and $\p^n$ are linear spaces on $X$. Thus we have:

\begin{cor}\label{corollarygeometric}
Let $X=V(f)\subset \p^N$ be a Nagata hypersurface of order $d_1$. There is a family of lines of dimension $m+n-1$ on $X$.
\end{cor}  
\begin{proof}
Let $p_\alpha\in\p^{m-1}$ be a point, then there is a family of lines of dimension $n$  that joins $p_\alpha$ and the linear space $\p^n$, for all $p_\alpha\in\p^{m-1}$. This family covers $\tilde{Y}_\alpha$. Then we have a family of lines of dimension $(n)+(m-1)=n+m-1$ on $X$. The singular locus of $X$ contains $\p^{m-1}$.
\par Conversely, let $\overline{p}\in\p^n$ be a point, then there is a family of lines of dimension $m-1$  that joins $\overline{p}$ and all points $q$ in the linear space $\p^{m-1}$. So the proof follows.
\end{proof}

\section{Simplicial Nagata idealization of order $k$}

\begin{defin}\label{sftdefn}\rm
A bihomogeneous polynomial 
\begin{equation}\label{eq:5}
f=\displaystyle\sum_{i=0}^n x_i^kg_i\in\K[x_0,\ldots,x_n,u_1,\ldots,u_m]_{(k,d-k)}\end{equation}
is called a simplicial Nagata polynomial of order $k$ if all $g_i$ are square free monomials.
\end{defin}

The following combinatorial constructions were inspired by \cite{GZ}.

\begin{defin}\label{csdefn}\rm
Let $V=\left\{u_1,\ldots,u_m\right\}$ be a finite set. A \emph{simplicial complex} $\Delta$ with vertex set $V$ is a collection of subsets of $V$, i.e. a subset of the power set $2^V$, such that for all $A\in\Delta$ and for all subset $B\subset A$, we have $B\in \Delta.$
\end{defin}

We say that $\Delta$ is a simplex if $\Delta=2^V.$\\
The members of $\Delta$ are referred as \emph{faces} and the maximal faces (respect to the inclusion) are the \emph{facets}. The vertex set of $\Delta$ is also called $0-$\emph{skeleton}. If $A\in \Delta$ and $|A|=k$, it is called a $(k-1)-$face, or a face of dimension $k-1$: the $0-$faces are the vertices and the $1-$faces are called \emph{edges}.
\begin{defin}
If all the facets have the same dimension $d>0$, the complex is said to be \emph{pure}.\end{defin}

Let $\Delta$ be a pure simplicial complex of dimension $d>0$ with vertex set $V=\left\{u_1,\ldots,u_m\right\}$, we denote by $f_k$ the number of $(k-1)-$faces, hence $f_0=1$, $f_1=m$, $f_{d+1}$ is the number of facets of $\Delta$ and $f_j=0$, for $j> d+1$.\\
\begin{rmk}\label{remarkimport}\rm
There is a natural bijection between the square free monomials, of degree $r$, in the variables $u_1,\ldots,u_m.$ and the $(r-1)$-faces of the simplex $2^V$, with vertex set $V=\left\{u_1,\ldots,u_m\right\}$. In fact, a square free monomial $g=u_{i_1}\cdots u_{i_r}$, in the variables $u_1,\ldots,u_m$, corresponds to the finite subset of $2^V$ given by $\left\{u_{i_1},\ldots,u_{i_r}\right\}.$ 9, to any finite subset $F$ of $2^V$, we associate the monomial $m_F=\displaystyle\prod_{u_i\in F}u_i$ of square free type.
\end{rmk}

An important result about simplicial Nagata idealization can be found in \cite[Theorem 3.2]{GZ}.

\subsection{Simplicial Nagata idealization of order $k$}
Let $f\in\K[x_0,\ldots,x_n,u_1,\ldots,u_m]_{(k,k+1)}$ be a simplicial Nagata polynomial of order $k$:  
\begin{equation}\label{eq:2}
f=\displaystyle \sum_{r=1}^nx_r^k g_r
\end{equation}
with $g_r$ monomials in variables $u_1,\ldots,u_m$ of degree $k+1$. 


We want to characterize the Hilbert vector of the algebras associated to the Nagata polynomial of type (\ref{eq:2}).
\\ Let $\Delta$ be a pure  simplicial complex of dimension $k$, with vertex set $V=\left\{u_1,\ldots,u_m\right\}$. We denote by $f_k$ the number of $(k-1)-$faces, hence $f_0=1, f_1=m, f_{k+1}$ is the number of the facets of $\Delta$ and $f_j=0$ for $j>k+1$.

The facets of $\Delta$, associated to $f$, corresponding to the monomials $g_i$, will be labeled by $g_i$.
The associated algebra is $A_\Delta=Q/\ann(f_\Delta)$. By abuse of notation, we will always denote $f_\Delta$ with $f$ and $A_\Delta$ with A.
\par  If $p\in \K[u_1,\ldots,u_m]$ is a square free monomial, we denote by $P\in\K[U_1,\ldots,U_m]$ the dual differential operator $P=p(U_1,\ldots,U_m).$

\begin{thm}\label{teoremaimportante}Let $f\in\K[x_0,\ldots,x_n,u_1,\ldots,u_m]_{(k,k+1)}$ be a simplicial Nagata polynomial of order $k$:  
\begin{equation*}
f=\displaystyle \sum_{r=1}^nx_r^k g_r
\end{equation*}
with $g_r$ monomials in variables $u_1,\ldots,u_m$ of degree $k+1$. 
Let $\Delta$ be a pure  simplicial complex of dimension $k$ and let $A=Q/ \ann(f)$. Then 
$$A=\displaystyle\bigoplus_{i=0}^{d=2k+1}A_i \mbox{ where } A_i=A_{(i,0)}\oplus A_{(i-1,1)}\oplus\cdots\oplus A_{(0,i)},\quad A_d=A_{(k,k+1)}$$
\begin{enumerate}
\item for all $j=1,\ldots,k+1:$
 $$\dim A_{(i,j)}=\begin{cases} f_j \quad\quad \mbox{ for }\quad  i=0\\
(n+1) \cdot \overline{f_j} \quad\quad \mbox{ for }\quad 1\leq i < k \\
f_{k+1-j} \mbox{ for } \quad i=k
\end{cases}$$ where $\overline{f_j}$ is the number of the subfaces, of dimension $j-1$, of the facet, $g_i,$  of $\Delta$.  
\item $I=\ann_Q(f)$ is generated by
\begin{itemize}

\item[(a)]$\left<X_0,\ldots,X_n\right>^{k+1}$ and $U_1^2,\ldots, U_m^2$;
\item[(b)]  the monomials in $I$ representing the minimal faces of the complement of $\Delta$, $\Delta^c$;
\item[(c)] the monomials $X_r^iP_r$, for $i=1,\ldots,k,$ such that, fixed the facet $M_r$ of $\Delta$, corresponding to the monomial $g_r$, $P_r$ is the dual differential operator of $p_r$; $p_r$ is a monomial in the variables $u_1,\ldots,u_m$, corresponding to a face $M'$ of $\Delta$ s.t. $M'\cap M_r=\emptyset;$
\item[(d)] the binomials $X_r^k\tilde{G}_r-X_s^k\tilde{G}_s$ where $g_r=\tilde{g}_rg_{rs}$ and $g_s=\tilde{g}_sg_{rs}$ and $g_{rs}$ represents a common subface of $g_r,g_s.$
\end{itemize}
\end{enumerate}
\end{thm}

\begin{proof}
\begin{enumerate}
\item Let $f$ be of type (\ref{eq:2}) associated to the pure  simplicial complex $\Delta$ of dimension $k$. The variables $u_1,\dots,u_m$ represent the vertices of $\Delta$.\\

We consider the following cases:
\begin{itemize}
\item for $i=0$ and $j=1,\ldots, k+1$, $A_ {(0,j)}$ is generated by the only monomials of degree $j$, in the variables $U_1,\ldots,U_{k+1}$, that do not annihilate $f.$ These monomials represent $(j-1)-$ faces of $\Delta$. We need to show that they are linearly independent over $\K.$\\ Consider $\left\{\Omega_1,\ldots,\Omega_{\nu}\right\}$ a system of monomials of $Q_{(0,j)}$, where $\Omega_s$, for $s=1,\ldots,\nu$, is associated to any $(j-1)-$ face $\omega$. We take any linear combination: $$0=\displaystyle\sum_{r=0}^{\nu}c_r\Omega_r(f)=\displaystyle\sum_{r=0}^{\nu}c_r\displaystyle\sum_{s=0}^nx_s^k\Omega_r(g_s)=\displaystyle\sum_{s=0}^nx_s\displaystyle\sum_{r=0}^{\nu}\Omega_r(g_s).$$ Therefore we get $\displaystyle\sum_{r=0}^{\nu}c_r\Omega_r(g_s)=0,$ for all $s=0,\ldots, n$. For each $r=0,\ldots,\nu$, there is a $s=0,\ldots,n$, such that if $\Omega_r(g_s)\neq 0$, then $c_r=0$ for all $r$. Hence  $\dim A_{(0,j)}=f_j,$ where $f_j$ is the number of $(j-1)-$faces of $\Delta$.
\item for $1\leq i < k$ and $j=1,\ldots, k+1$, the generators of $A_{(i,j)}$ are the monomials of type $X^i_sU_{r_1}U_{r_2}\cdots U_{r_j}$ for $s=0,\ldots,n$, for all $j$. Fix $s=0,\ldots,n$, and let $M_s$ be the facet of $\Delta,$ corresponding to the monomial $g_s$, the monomial $U_{r_1}U_{r_2}\cdots U_{r_j}$ of $Q_{(0,j)}$ is the dual differential operator of the monomial $u_{r_1}u_{r_2}\cdots u_{r_j}$, that gives the $(j-1)-$dimensional subfaces of $M_s$. The monomials $X^i_sU_{r_1}U_{r_2}\cdots U_{r_j}$ for $s=0,\ldots,n$, for all $j$ are linearly independent. In fact, denoting by $\Omega_s^i$ the monomial $X^i_sU_{r_1}U_{r_2}\cdots U_{r_j}$, for $s=0,\ldots,n$, we note that: $$\Omega_s^i(f)=cx_s^{k-i}(U_{r_1}\cdots U_{r_j})(g_s)\neq 0$$ since $(U_{r_1}\cdots U_{r_j})(g_s)$ identifies the vertices of the $(j-1)-$dimensional face. We get: $$\displaystyle\sum_{s=0}^nc_s\Omega_s^i(f)=0 \Leftrightarrow c_s=0 \quad \forall s.$$
For $s=0,\ldots,n$, in correspondence of $\Omega_s^i(f)$, we can get a number of $(j-1)-$dimensional faces of $\Delta$. Denoting such number by $\overline{f}_j$, we have $\dim A_{(i,j)}=(n+1)\cdot \overline{f_j}$.  
\item for $i=k$ and $j=1,\ldots, k$, by duality $A^*_{(0,k+1-j)}\simeq A_{(k,j)} $, thence we have: $$\dim A_{(k,j)}=\dim A^*_{(0,k+1-j)}=f_{k+1-j}.$$
\end{itemize}
 \item Let $I=\ann(f)$ be the annihilator. We consider the following exact sequence:
\begin{equation}\label{eq:3}
 \begin{CD}
  0@>>>  I_{(i,j)}@>>>  Q_{(i,j)} @>>> A_{(k-i,k+1-j)} @>>> 0.
 \end{CD}\end{equation} we have the following cases:

\begin{itemize}
\item for $i=0$ and $1\leq j\leq k+1$,  we have by (\ref{eq:3}) $$\dim A_{(0,j)}=f_j\Rightarrow \dim I_{(0,j)}=\dim Q_{(0,j)}-f_j.$$ Since $A_{(0,j)}$ has a basis given by the $(j-1)-$faces of $\Delta$, then $I_{(0,j)}$ is generated by monomials representing all the $(j-1)-$faces of the complement of $\Delta$. In $I$, it is enough to consider the minimal faces of $\Delta^c$, by definition of ideal. \\ We note that in $I_{(0,2)}$ there are also the monomials $U_1^2,\ldots, U_m^2$, since the monomials $g_i$, in the variables $u_1,\ldots,u_m$ are square free.
\item for $1\leq i<k$ and $1\leq j\leq k+1$, fix the facet $M_r$ of $\Delta$, corresponding to $g_r$, since $A_{(i,j)}$ has a basis given by the $(j-1)-$dimensional subfaces of $M_r$,  then $I_{(i,j)}$ is generated by monomials $X^i_rP_r$ where $P_r$ is the dual differential operator of $p_r$; $p_r$ is a monomial in the variables $u_1,\ldots,u_m$, corresponding to a $(j-1)-$dimensional face $\overline{M}_r$ s.t. $\overline{M}_r\in\Delta^c$ or $\overline{M}_r\in\Delta$ and $\overline{M}_r\cap M_r=\emptyset$.
\item for $i=k$ and $1\leq j\leq k+1$, we fix two facets of $\Delta$, $M_r$ and $M_s$, corresponding to the monomials $g_r$ and $g_s$, and such that $M_r\cap M_s\neq\emptyset$. Let $M_{rs}=M_r\cap M_s$; we denote the monomial corresponding to it  by $g_{rs}$. We consider $\tilde{M}_r=M_r\textbackslash M_{rs}$ and $\tilde{M}_s=M_s\textbackslash M_{rs}$. Let $\tilde{g}_r$ and $\tilde{g}_s$ be the monomials corresponding to $\tilde{M}_r$ and $\tilde{M}_s$. We note that $\tilde{M}_r\cap \tilde{M}_s=\emptyset$. Hence the binomials, $X_r^k\tilde{G}_r-X_s^k\tilde{G}_s$, are in $I_{(k,j)}$, where $\tilde{G}_r$ and $\tilde{G}_s$ are the dual differential operators of $\tilde{g}_r$ and $\tilde{g}_s$ respectively.\\ Let us consider the following exact sequence: \begin{equation*}\label{eq:3}
 \begin{CD}
  0@>>>  I_{(k,j)}@>>>  Q_{(k,j)} @>>> A_{(0,k+1-j)} @>>> 0,
 \end{CD}\end{equation*} we get $\dim I_{(k,j)}=\dim Q_{(k,j)}-f_{k+1-j}.$ Let $\tilde{Q}_{(k,j)}$ be the $\K-$space spanned by all the monomials $X^k_r\tilde{G}_r$, where $\tilde{G}_r$ is the dual differential operator of $g_r$ that is a monomial in the variables $u_1,\ldots,u_m$, corresponding to a subface of $M_r$. Let $\overline{I}_{(k,j)}\subset I_{(k,j)}$ be the $\K-$vector space spanned by the monomials $X^k_rP_r$, where $P_r$ is the dual differential operator of the monomial, in the variables $u_1,\ldots,u_m$, $p_r$, not corresponding to a subface of $M_r$. They are two $\K-$vector spaces s.t. $Q_{(k,j)}=\tilde{Q}_{(k,j)}\oplus\overline{I}_{(k,j)}.$ We consider the ideal $\tilde{I}_{(k,j)}\subset\tilde{Q}_{(k,j)}$. The exact sequence given by evaluation restricted to $\tilde{Q}_{(k,j)}$ becomes:
\begin{equation*}
 \begin{CD}
  0@>>> \tilde{I}_{(k,j)}@>>>  \tilde{Q}_{(k,j)} @>>> A_{(0,k+1-j)} @>>> 0.
 \end{CD}\end{equation*} We note: \begin{multline*}
\dim I_{(k,j)}=\dim Q_{(k,j)}-f_{k+1-j}=\dim\tilde{Q}_{(k,j)}+\dim\overline{I}_{(k,j)}-f_{k+1-j}=\\=\dim\tilde{I}_{(k,j)}+f_{k+1-j}+\dim\overline{I}_{(k,j)}-f_{k+1-j}=\dim\tilde{I}_{(k,j)}+\dim\overline{I}_{(k,j)}.\end{multline*} Hence $I_{(k,j)}=\tilde{I}_{(k,j)}\oplus\overline{I}_{(k,j)}$. The generators of $\tilde{I}_{(k,j)}$ are the binomial $X_r^k\tilde{G}_r-X_s^k\tilde{G}_s$ precisely. The result follows.


\end{itemize}
Moreover for $i=k+1$ and $j=0$, it is clear that $I_{(k+1,0)}=(X_0,\ldots, X_n)^{k+1}$. In fact $X_i^{k+1}(f)=0$, for $i=0,\ldots,n$, since the monomials in $x_0,\ldots,x_n$ of $f$ have degree $k$, by Remark \ref{rmk1.8}.
\end{enumerate}
\end{proof}
We discuss the following example:

\begin{exm}\label{ex2} \rm
Let $V=\{u_1,\ldots,u_6\}$ be a finite set. We have: $$2^V=\{\emptyset,\{u_1\},\ldots,\{u_6\},\ldots, \{u_1,\ldots,u_6\}\}$$
Let $\Delta$ be the following simplicial complex: $$\Delta=\{\emptyset,\underbrace{\{u_1\},\ldots,\{u_6\}}_{\text{ vertices }},\underbrace{\{u_1,u_2\},\ldots,\{u_5,u_6\}}_{\text{ edges }},\underbrace{\{u_1,u_2,u_3\},\ldots,\{u_2,u_3,u_6\}}_{\text{ $2-$faces }}\}$$ It is given by two pyramids, with the common basis, of vertices $u_1,\dots u_5$ and $u_6$ and faces labeled by $g_0,\dots,g_6$ and $g_7$:
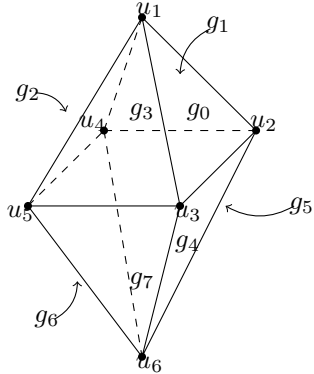
\begin{center}
\begin{tikzpicture}[scale=0.50]
\path (0,0) edge[dashed] (2,2);
\path (2,2) edge[dashed] (6,2);
\draw (0,0)--(4,0)--(6,2);
\draw (4,0)--(3,-4)--(6,2);
\draw (4,0)--(3,5)--(6,2);
\path (0,0) edge (3,5);
\path (0,0) edge (3,-4);
\path (2,2) edge[dashed] (3,-4);
\path (2,2) edge[dashed] (3,5);
\fill (0,0) circle(3pt);
\fill (2,2) circle(3pt);
\fill (4,0) circle(3pt);
\fill (6,2) circle(3pt); 
\fill (3,5) circle(3pt);
\fill (3,-4) circle(3pt);
\node at (3.2,5.2) {$u_1$};
\node at (6.2,2.2){$u_2$};
\node at (4.2,-0.2){$u_3$};
\node at (1.7,2.2){$u_4$};
\node at (-0.2,-0.2){$u_5$};
\node at (3.2,-4.2){$u_6$};
\node at (4.5,2.5){$g_0$};
\node at (5,4.8){$g_1$};
\node at (0,3){$g_2$};
\node at (3,2.5) {$g_3$};
\node at (4.2,-1) {$g_4$};
\node at (7.2,0) {$g_5$};
\node at (0.5,-3){$g_6$};
\node at (3,-2){$g_7$};
\path (4.8,4.7) edge[->,bend right] (4,3.5);
\path (0.2,3) edge[->, bend left] (1,2.5);
\path (7,0) edge[->,bend left] (5.2,0);
\path (0.7,-3) edge[->,bend right] (1.3,-2);
\end{tikzpicture}
\end{center}
The $2-$faces are the facets of $\Delta$, then $\Delta$ is pure  of dimension $2$.\\
Let \begin{multline*}
f=f_\Delta=x_0^2u_1u_2u_3+x_1^2u_1u_2u_4+x_2^2u_1u_4u_5+x^2_3u_1u_3u_5+\\
+x^2_4u_2u_3u_6+x^2_5u_2u_4u_6+x^2_6u_4u_5u_6+x^2_7u_3u_5u_6\end{multline*} be the bihomogeneous polynomial of degree $5$. It is a Nagata polynomial of order $2$ and the monomials $g_0=u_1u_2u_3$, $g_1=u_1u_2u_4$, $g_2=u_1u_4u_5$, $g_3=u_1u_3u_5$, $g_4=u_2u_3u_6$, $g_5=u_2u_4u_6$, $g_6=u_4u_5u_6$ and $g_7=u_3u_5u_6$ are of square free type.

We have $$A=A_0\oplus A_1 \oplus A_2 \oplus A_3 \oplus A_4\oplus A_5$$ and the Hilbert vector is given by: $$h_0=1=h_5 \mbox { and } h_1=14=h_4.$$ We calculate $h_2=\dim A_2$ and $h_3=\dim A_3$.\\ By Theorem \ref{teoremaimportante}, we have 
\begin{equation*}
\begin{split}
h_2&=\dim A_2=\dim A_{(2,0)}+\dim A_{(1,1)}+\dim A_{(0,2)}=\\
&=f_3+8\cdot \overline{f_1}+f_2=8+8\cdot 3+12=44
\end{split}
\end{equation*}
\begin{equation*}
\begin{split}
h_3&=\dim A_3=\dim A_{(3,0)}+\dim A_{(2,1)}+\dim A_{(1,2)}+\dim A_{(0,3)}=\\
&=0+f_2+ 8\cdot \overline{f_2} + f_3=12+8\cdot 3+8=44
\end{split}
\end{equation*}
Hence the Hilbert vector is $(1,14,44,44,14,1).$\\

By Theorem \ref{teoremaimportante}, $I=\ann(f)$ is generated by:
\begin{itemize}
\item $\langle X_0,\ldots,X_7\rangle^3$ and $U_1^2, \ldots, U_m^2$, by the part (2a);
\item since the complement of $\Delta$ is:
\begin{multline*}\Delta^c=\{\underbrace{\left\{u_1,u_6\right\},\ldots,\left\{u_2,u_5\right\}}_{\text{diagonals}},\underbrace{\left\{u_1,u_2,u_6\right\},\ldots,\left\{u_1,u_5,u_6\right\}}_{\text{$2-$faces}},\\ \underbrace{\left\{u_1,u_2,u_3,u_5\right\},\ldots,\left\{u_2,u_3,u_5,u_6\right\}}_{\text{$3-$faces}},\ldots,\left\{u_1,\ldots,u_6\right\}\}\end{multline*}
and since diagonals are the minimal faces of $\Delta^c$, then the monomials $U_1U_6$, $U_3U_4$ and $U_2U_5$ are in $I=\ann(f)$, by the part (2b). 
\item For $i=1,2$, fix the facet $M_0=\{u_1,u_2,u_3\}\in\Delta$, corresponding to the monomial $g_0$, we have that the monomial $p_0$ represents:
\begin{itemize}
\item one of the remaining vertices, for example $u_4$:
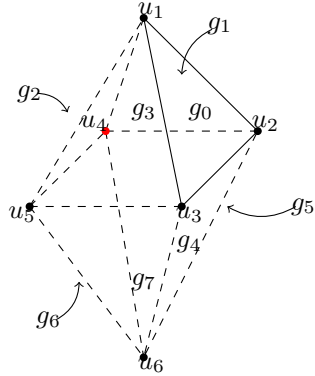
\begin{center}
\begin{tikzpicture}[scale=0.50]
\path (0,0) edge[dashed] (2,2);
\path (2,2) edge[dashed] (6,2);
\draw[dashed] (0,0)--(4,0);
\draw (4,0) edge (6,2);
\draw[dashed] (4,0)--(3,-4)--(6,2);
\draw (4,0)--(3,5)--(6,2);
\path (0,0) edge[dashed] (3,5);
\path (0,0) edge[dashed] (3,-4);
\path (2,2) edge[dashed] (3,-4);
\path (2,2) edge[dashed] (3,5);
\fill (0,0) circle(3pt);
\fill[red] (2,2) circle(3pt);
\fill (4,0) circle(3pt);
\fill (6,2) circle(3pt); 
\fill (3,5) circle(3pt);
\fill (3,-4) circle(3pt);
\node at (3.2,5.2) {$u_1$};
\node at (6.2,2.2){$u_2$};
\node at (4.2,-0.2){$u_3$};
\node at (1.7,2.2){$u_4$};
\node at (-0.2,-0.2){$u_5$};
\node at (3.2,-4.2){$u_6$};
\node at (4.5,2.5){$g_0$};
\node at (5,4.8){$g_1$};
\node at (0,3){$g_2$};
\node at (3,2.5) {$g_3$};
\node at (4.2,-1) {$g_4$};
\node at (7.2,0) {$g_5$};
\node at (0.5,-3){$g_6$};
\node at (3,-2){$g_7$};
\path (4.8,4.7) edge[->,bend right] (4,3.5);
\path (0.2,3) edge[->, bend left] (1,2.5);
\path (7,0) edge[->,bend left] (5.2,0);
\path (0.7,-3) edge[->,bend right] (1.3,-2);
\end{tikzpicture}
\end{center}
and finally we get: $P_0=p_0(U_1,\ldots,U_8)=U_4$.
Thence the monomial of degree $i+1$, $X_0^{i}U_4$ is in $I=\ann(f)$. The other monomials of this type are obtained with the same procedure.
\item one of the remaining edges, for example the edge that joins the vertices $u_5$ and $u_6$:
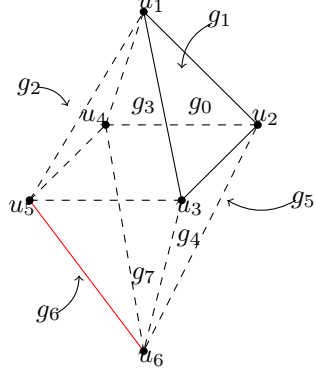
\begin{center}
\begin{tikzpicture}[scale=0.50]
\path (0,0) edge[dashed] (2,2);
\path (2,2) edge[dashed] (6,2);
\draw[dashed] (0,0)--(4,0);
\draw (4,0) edge (6,2);
\draw[dashed] (4,0)--(3,-4)--(6,2);
\draw (4,0)--(3,5)--(6,2);
\path (0,0) edge[dashed] (3,5);
\path (0,0) edge[red] (3,-4);
\path (2,2) edge[dashed] (3,-4);
\path (2,2) edge[dashed] (3,5);
\fill (0,0) circle(3pt);
\fill (2,2) circle(3pt);
\fill (4,0) circle(3pt);
\fill (6,2) circle(3pt); 
\fill (3,5) circle(3pt);
\fill (3,-4) circle(3pt);
\node at (3.2,5.2) {$u_1$};
\node at (6.2,2.2){$u_2$};
\node at (4.2,-0.2){$u_3$};
\node at (1.7,2.2){$u_4$};
\node at (-0.2,-0.2){$u_5$};
\node at (3.2,-4.2){$u_6$};
\node at (4.5,2.5){$g_0$};
\node at (5,4.8){$g_1$};
\node at (0,3){$g_2$};
\node at (3,2.5) {$g_3$};
\node at (4.2,-1) {$g_4$};
\node at (7.2,0) {$g_5$};
\node at (0.5,-3){$g_6$};
\node at (3,-2){$g_7$};
\path (4.8,4.7) edge[->,bend right] (4,3.5);
\path (0.2,3) edge[->, bend left] (1,2.5);
\path (7,0) edge[->,bend left] (5.2,0);
\path (0.7,-3) edge[->,bend right] (1.3,-2);
\end{tikzpicture}
\end{center}
We get: $$P_0=p_0(U_1,\ldots,U_8)=U_5U_6$$ and the monomial of degree $i+2$, $X_0^iU_5U_6$, is in $I=\ann(f)$, by part (2c). The other monomials of this type are obtained by the same way.
\end{itemize}
\item The faces $g_0$ and $g_3$ have the common edge that joins the vertices $u_1u_3:$
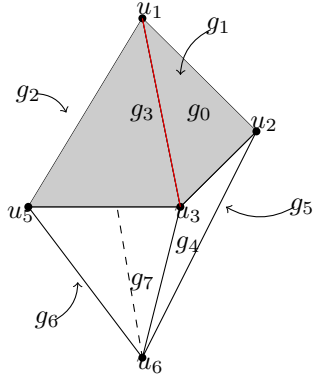
\begin{center}
\begin{tikzpicture}[scale=0.50]
\path (0,0) edge[dashed] (2,2);
\path (2,2) edge[dashed] (6,2);
\draw (0,0)--(4,0)--(6,2);
\draw (4,0)--(3,-4)--(6,2);
\draw (4,0)--(3,5)--(6,2);
\path (0,0) edge (3,5);
\path (0,0) edge (3,-4);
\path (2,2) edge[dashed] (3,-4);
\path (2,2) edge[dashed] (3,5);
\draw [fill=black!20] (0,0)--(4,0)--(3,5);
\draw [fill=black!20] (3,5)--(4,0)--(6,2);
\fill (0,0) circle(3pt);
\fill (4,0) circle(3pt);
\fill (6,2) circle(3pt); 
\fill (3,5) circle(3pt);
\fill (3,-4) circle(3pt);
\node at (3.2,5.2) {$u_1$};
\node at (6.2,2.2){$u_2$};
\node at (4.2,-0.2){$u_3$};
\node at (-0.2,-0.2){$u_5$};
\node at (3.2,-4.2){$u_6$};
\node at (4.5,2.5){$g_0$};
\node at (5,4.8){$g_1$};
\node at (0,3){$g_2$};
\node at (3,2.5) {$g_3$};
\node at (4.2,-1) {$g_4$};
\node at (7.2,0) {$g_5$};
\node at (0.5,-3){$g_6$};
\node at (3,-2){$g_7$};
\path (4.8,4.7) edge[->,bend right] (4,3.5);
\path (0.2,3) edge[->, bend left] (1,2.5);
\path (7,0) edge[->,bend left] (5.2,0);
\path (0.7,-3) edge[->,bend right] (1.3,-2);
\path (3,5) edge[red] (4,0);
\end{tikzpicture}
\end{center}
$g_{1,3}$ represents the edge that joins the vertices $u_1$ and $u_3$. $\tilde{g_0}$ and $\tilde{g_3}$ represent the vertices $u_2$ and $u_5$ respectively. We have: $$\tilde{G_0}=\tilde{g_0}(U_1,\ldots,U_6)=U_2 \mbox{ and } \tilde{G_3}=\tilde{g_3}(U_1,\ldots,U_6)=U_5.$$ The binomial, of degree $3$, $X_0^2U_2-X_3^2U_5,$ is in $I=\ann(f)$ by the part (2d). The other binomials of degree $3$ of this type are obtained by the same procedure.  
We note that the faces $g_0$ and $g_2$ have the common vertex $u_1$, hence, in the ideal $I=\ann (f)$ there is the binomial, of degree $4$, $X_0^2U_2U_3-X_2^2U_4U_5;$ the other binomials, of degree $4$, are obtained by the same procedure.
\end{itemize}
\end{exm}
{\bf Acknowledgments}.

The second author would thank M. Boij, J. Migliore, U. Nagel, A. Seceleanu, H. Schenck and J. Watanabe for the very stimulating discussions about the subject in Banff. I would like to thank 
also all the organizers and the participants of workshop {\it Lefschetz Properties and Artinian Algebras} at BIRS, Banff, Canada in Mach, 2016. It was a very pleasant time to do Mathematics and enjoy a very nice place. \\
The second author thanks Carolina Araujo for useful discussions about the geometry of Nagata hypersurfaces, comparing it with the classical case. \\
The authors want to thank P. de Poi for his import remarks and comments on a previous version of the manuscript.\\

The second author was partially supported  by ICTP-INdAM Research in Pairs Fellowship 2018/2019 and by FACEPE ATP - 0005-1.01/18.

\end{document}